\documentclass{amsart}
\usepackage{amsfonts, amsmath, amssymb, mathrsfs, amscd}
\usepackage[active]{srcltx}
\usepackage{verbatim}

\theoremstyle{plain}
\newtheorem{theorem}{Theorem}[section]
\theoremstyle{remark}
\newtheorem{remark}[theorem]{Remark}
\newtheorem{example}[theorem]{Example}
\theoremstyle{plain}

\newtheorem{lemma}[theorem]{Lemma}
\newtheorem{proposition}[theorem]{Proposition}
\newtheorem{definition}[theorem]{Definition}

\numberwithin{equation}{section}

\newcommand{\R}{\mathbb{R}}

\newcommand{\nN}{n \in \mathbb{N}}
\newcommand{\N}{\mathbb{N}}
\newcommand{\C}{\mathbb{C}}
\newcommand{\K}{\mathbb{K}}
\newcommand{\E}{\mathbb{E}}

\newcommand{\bean}{\begin{eqnarray*}}
\newcommand{\eean}{\end{eqnarray*}}

\newcommand{\la}{\langle}
\newcommand{\ra}{\rangle}
\newcommand{\lb}{\langle}
\newcommand{\rb}{\rangle}
\newcommand{\wh}{\widehat}

\newcommand{\n}{\Vert}
\newcommand{\s}{^*}
\newcommand{\e}{\varepsilon}

\newcommand{\embed}{\hookrightarrow}
\renewcommand{\odot}{\hbox{\tiny\textcircled{s}}}

\newcommand{\calB}{\mathscr{B}}
\newcommand{\calF}{\mathscr{F}}

\newcommand{\calL}{\mathscr{L}}

\newcommand{\calY}{\mathscr{Y}}
\renewcommand{\P}{\mathbb{P}}

\newcommand{\one}{\mathbf{1}}

\renewcommand{\H}{{\mathscr H}}

\allowdisplaybreaks

\begin{document}

\title[Second quantisation for skew convolution products] {Second Quantisation
for Skew Convolution Products of Measures in Banach Spaces}

\author{David Applebaum}

\email{D.Applebaum@sheffield.ac.uk}
\address{School of Mathematics and Statistics,\\
University of Sheffield, \\
Sheffield S3 7RH\\
United Kingdom.}

\author{Jan van Neerven}
\email{J.M.A.M.vanNeerven@tudelft.nl}
\address{Delft Institute of Applied Mathematics,\\
Delft University of Technology,\\
PO Box 5031\\ 2600 GA Delft, \\
The Netherlands.}

\date{}

\begin{abstract} We study measures in Banach space which arise as the skew
convolution
product of two other measures where the convolution is deformed by a skew map.
This is the structure that underlies both the theory of Mehler semigroups and
operator
self-decomposable measures. We show how that given such a set-up the skew map
can be
lifted to an operator that acts at the level of function spaces and demonstrate
that
this is an example of the well known functorial procedure of second
quantisation. We
give particular emphasis to the case where the product measure is infinitely
divisible
and study the second quantisation process in some detail using chaos expansions
when
this is either Gaussian or is generated by a Poisson random measure.
\end{abstract}

\maketitle

\section{Introduction}

In recent years there has been considerable interest in skew-convolution
semigroups of probability measures in Banach spaces and the so-called Mehler
semigroups that they induce on function spaces. These objects arise naturally in
the study of infinite dimensional Ornstein-Uhlenbeck processes driven by
Banach-space valued L\'{e}vy processes. Such processes have attracted much
attention as they are the solutions of the simplest non-trivial class of
stochastic partial differential equations driven by additive L\'{e}vy noise (see
\cite{App1, CM, PZ}). The first systematic study of Mehler semigroups in their
own right were \cite{BRS} and \cite{FR} with the former concentrating on
Gaussian noise while the latter generalised to the L\'{e}vy case. Harnack
inequalities were obtained in \cite{RW} and the infinitesimal generators were
found in \cite{App4}. From a different point of view, skew-convolution
semigroups also appear naturally in the investigation of continuous state
branching processes with immigration \cite{DLSS} and more general affine
processes \cite{DL}.

In this paper we focus on the representation of Mehler semigroups as second
quantised operators. Such a result has been known for a long time in the
Gaussian case. It was first established for Hilbert space valued semigroups in
\cite{CMG1} and then extended to Banach spaces in \cite{VN1}. Once such a
representation is known it can be put to good use in proving key properties of
the semigroup such as compactness and smoothness \cite{CMG1},
symmetry \cite{CMG2}, analyticity \cite{GolNee, MN},
and in the computation of their $L^{p}$ spectra \cite{VN2}.
When the semigroups act on Hilbert spaces, the desired second quantisation
representation was
recently obtained in \cite{Pesz} in the pure jump case using chaotic
decomposition techniques from \cite{LP}, under the assumption that the Ornstein-Uhlenbeck
process has an invariant measure. This paper extends that result to
the Banach space case and obtains the second quantisation representation without
needing to assume the
existence of an invariant measure.

In fact, within the main part of our paper we dispense with Mehler semigroups
altogether and work with a more general structure which we introduce
herein. For this
we require that there are measures
$\mu_{1}$ on a Banach space $E_1$ and $\mu_{2}$ and $\rho$
on a Banach space
 $E_2$ which are related by the identity
$$ \mu_{2} = T(\mu_{1}) * \rho,$$ where $T:E_1\to E_2$ is a Borel mapping and $*$
is the usual convolution of measures. An operator $T$ that has such an induced
action is precisely a {\it skew map} as featured in the abstract
of this paper. Note that if $E_1 = E_2 = E$ and $\mu_{1} = \mu_{2} = \mu$ say,
then $\mu$ is an {\it
operator self-decomposable} measure and such objects have been intensely studied
(see e.g. \cite{Ju1, JV, Urb}.) The invariant measures arising in \cite{Pesz}
are
precisely of this form. On the other hand a skew convolution semigroup of
measures $(\mu_{t}, t \geq 0)$ with respect to a $C_{0}$-semigroup $(S(t), t
\geq 0)$ is characterised by the relations $\mu_{s+t} = S(t)\mu_{s} * \mu_{t}$
and these are clearly also examples of our structure. At our more general
level, the antecedent of a Mehler semigroup is a bounded linear operator $P_{T}$
which acts from $L^{2}(E_2,\mu_{2})$ to $L^{2}(E_1, \mu_{1})$. Our main result
is then to show that this operator can be seen as a second quantisation
of the adjoint $T\s: E_2\to E_1$ in a
natural way in the case where $\mu_{1}$ and $\mu_{2}$ are both infinitely
divisible and either Gaussian or of pure jump type.

A key part of our approach is the use of a family of vectors that we call {\it
exponential martingale vectors}. We now explain how these arise and contrast
them with
the more familiar {\it exponential vectors} (see e.g. \cite{App3, Par}). Second
quantisation
is seen most naturally as a covariant functor $\Gamma$ within the category whose
objects are Hilbert spaces and morphisms are contractions (see e.g. \cite{Par}).
If $H$ is a Hilbert space and $\Gamma(H)$ is the associated symmetric Fock
space, the set of exponential vectors is linearly independent and total in Fock
space. If we are given a Gaussian field over $H$ then the exponential vectors
correspond to the generating functions of the Hermite polynomials, and from the
point of view of stochastic calculus they correspond both to the
Dol\'{e}ans-Dade
exponentials and to the exponential martingales. When we consider L\'{e}vy
processes, the latter symmetry is broken. Exponential vectors still correspond
to Dol\'{e}ans-Dade exponentials (see \cite{App3}) but these are no longer
exponential
martingales. In this paper, we find that a natural context for defining second
quantisation in a non-Gaussian context is to employ vectors that are natural
generalisations of exponential martingales, rather than using exponential
vectors themselves. Hence we call these exponential martingale vectors. In
particular, as we show
in Section \ref{sec:skew} and the appendix, these are still both total and
linearly
independent.

\vspace{5pt}

{\it Notation}. Throughout this article, $E$ is a real Banach space.
The space of all bounded linear operators on $E$ is denoted by $\calL(E)$
and the dual of $E$ is denoted by $E^{*}$.
The action of $E^{*}$ on $E$ is represented by $x\s(x) = \la x, x\s \ra$.
Whenever we consider measures
on a Banach space $E$, they are defined on the Borel $\sigma$-algebra
${\calB}(E)$.
If $\mu$ is a Borel measure on $E$ and $T: E\to F$ is a Borel mapping from $E$
into another Banach space $F$ we frequently write $T(\mu)$ to denote the Borel
measure
$\mu \circ T^{-1}$.
The Dirac measure based at $x \in E$
is denoted by $\delta_{x}$. The Banach space (with respect to the supremum norm)
of all bounded Borel measurable functions on $E$ will be denoted
$B_{\rm b}(E;\K)$, where $\K$ is either $\R$ or $\C$. If both choices
are permitted we simply write $B_{\rm b}(E)$.

\section{Skew convolution of measures and associated skew maps}\label{sec:skew}

Let $\nu$ be a {\em finite Radon measure} on a Banach space $E$, that is,
$\nu$ is a finite Borel measure on $E$
with the property that for all $\e>0$ there exists a compact set $K$ in $E$
such that $\nu(E\setminus K) <\e.$ Recall that if $E$ is separable, then
every finite Borel measure is Radon.

The {\it characteristic function} of $\nu$
is the mapping $\widehat{\nu}:E^{*} \to \C$ defined by
$$ \widehat{\nu}(x\s) = \int_{E}\exp(i\langle x, x\s \rangle)\, \nu(dx),$$
for all $x\s \in E^{*}$.
The mapping $\widehat{\nu}$ is continuous with respect to the topology of
uniform convergence on
compact subsets of $E$. More generally, for a measurable function $\phi:E\to \R$
we may define
$$  \widehat{\nu}(\phi) = \int_{E}\exp(i\phi(x))\, \nu(dx).$$

\begin{definition}\label{def:skew}
Let $\mu_1$ and $\mu_2$ be Radon probability measures on the Banach spaces $E_1$
and $E_2$,
respectively, with
$\wh{\mu_2}(x\s)\not=0$ for all $x\s\in E_2\s$ (e.g. this condition is fulfilled,
when $\mu_{2}$ is infinitely divisible).
A Borel mapping $T:E_1\to E_2$ is called a {\em skew map} with respect to the
pair
$(\mu_1,\mu_2)$ if there exists a Radon probability measure $\rho$ on $E_2$ such
that
$$ T(\mu_1) * \rho = \mu_2,$$ and we say that $\mu_{2}$ is the {\em
skew-convolution product}
(with respect to $T$) of $\mu_{1}$ and $\rho$.
If $T$ is also a bounded linear operator between $E_{1}$ and $E_{2}$ we call it
a {\em skew operator} with respect to $(\mu_1,\mu_2)$.
\end{definition}

Given the pair $(\mu_1,\mu_2)$, the measure $\rho$ is easily seen to be unique.
Indeed, the identity
$ \wh {T(\mu_1)}(x\s) \wh\rho(x\s) = \wh{\mu_2}(x\s) \not =0$
forces $\wh {T(\mu_1)}(x\s)\not=0$, and
therefore $\wh \rho $ is uniquely determined by $T(\mu_1)$ and $\mu_2$.
We call $\rho$ the {\em skew convolution factor} associated with $T$ and the
pair $(\mu_1,\mu_2)$.

\begin{proposition} \label{prop:Lp}
Suppose that $T:E_1\to E_2$ is a skew map with respect to the pair
$(\mu_1,\mu_2)$, where
$\wh{\mu_2}(x\s)\not=0$ for all $x\s\in E_2\s$. Let $\rho$ be the associated
skew convolution factor.
For all $1\le p<\infty$ the linear mapping $P_T: B_{\rm b}(E_2)\to B_{\rm
b}(E_1)$ defined by
$$ P_T f(x) := \int_{E_2} f(T(x)+y)\,d\rho(y), \quad x\in E_1,$$
extends uniquely to a linear contraction $P_T: L^p(E_2,\mu_2)\to
L^p(E_1,\mu_1)$.
\end{proposition}
\begin{proof}
Fix $1\le p<\infty$. By the H\"older inequality, for all $f\in B_{\rm b}(E_2)$
we have
\begin{align*}
\n P_T f\n_{L^p(E_1,\mu_1)}^p
   & = \int_{E_1} \Big| \int_{E_2} f(T(x)+y)\,d\rho(y)\Big|^p\, d\mu_1(x)
\\ & \le  \int_{E_1} \int_{E_2} |f(T(x)+y)|^p\,d\rho(y)\, d\mu_1(x)
\\ & = \int_{E_2} \int_{E_2} |f(y'+y)|^p\,d\rho(y)\, dT(\mu_1)(y')
\\ & = \int_{E_2} |f(z)|^p \, d(T(\mu_1)*\rho)(z)
\\ & = \int_{E_2} |f(z)|^p \, d\mu_2(z)
= \n f\n_{L^p(E_2,\mu_2)}^p,
\end{align*}
and the required result follows.
 \end{proof}

\begin{example}[Skew Convolution Semigroups]\label{ex:1}
Let $(S(t), t \geq 0)$ be a $C_{0}$-semigroup
on a Banach space $E$. A skew convolution semigroup is a family $(\mu_{t}, t
\geq 0)$ of Radon probability
measures on $E$ for which $\mu_{s+t} = S(t)\mu_{s}*\mu_{t}$ for all $s,t \geq
0$. Then $S(t)$ is
a skew operator with respect to the pair $(\mu_{s}, \mu_{s+t})$, In this case we
write $P_{t}$ for
the linear operator $P_{S(t)}$. Then $(P_{t}, t \geq 0)$ is a semigroup in that
$P_0 = I$ and $P_{s+t} = P_{s}P_{t}$
for all $s, t \geq 0$, and is called a {\it Mehler semigroup} (see e.g.
\cite{App4, BRS, DL, DLSS, FR}). Such objects arise naturally in the study of
linear stochastic partial differential equations with additive
noise of the form:
\begin{align} \label{Lan}
 dY(t) = AY(t) + dL(t),
\end{align}
 where $A$ is the infinitesimal generator of $(S(t), t \geq 0)$ and $(L(t), t
\geq 0)$ is an $E$-valued
L\'{e}vy process. If $E$ is a real Hilbert space then it is well-known (see e.g.
\cite{App1, CM} and the
recent book \cite{PZ}) that
this equation has a unique mild (equivalently weak) solution $(Y(t), t \geq 0)$
which is a Markov process
given by the {\it generalised Ornstein-Uhlenbeck process}:
\begin{align} \label{OUproc}
Y(t) = S(t)Y(0) + \int_{0}^{t}S(t-u) \,dL(u),
\end{align}
(where the initial condition $Y(0)$ is assumed to be independent of $(L(t), t
\geq 0)$.)
Then $\mu_{t}$ is the law of the $E$-valued random variable $\int_{0}^{t}
S(t-u)\,dL(u)$ and
$(P_{t}, t \geq 0)$ is the transition semigroup of $(Y(t), t \geq 0)$.
On a Banach space we may define the stochastic convolution in (\ref{OUproc}) by
using integration
by parts as in \cite{JV}. Quite general necessary and sufficient conditions for
solutions to exist
to (\ref{Lan}) (where the stochastic convolution is defined in the sense of
It\^{o} calculus) are
given in \cite{RvG}. If $X$ is a Brownian motion, we refer
the reader to \cite{vNW}.
\end{example}

\begin{example}[Operator Self-Decomposable Measures]
Let $\mu$ be a Radon probability measure on $E$ that takes the form
\begin{align} \label{SDM}
 \mu = T\mu * \rho,
 \end{align}
where $T$ is a bounded linear operator on $E$ and $\rho$ is another Radon
probability measure on $E$. Then $\mu$
is {\it operator self-decomposable} (see \cite{Urb}) and $T$ is a skew operator
with
respect to the pair $(\mu, \mu)$. There has been extensive work on such measures
in the case where (\ref{SDM})
 holds with $T = S(t)$ for all $t \geq 0$ where $(S(t), t \geq 0)$ is a
$C_{0}$-semigroup on $E$
 (see e.g. \cite{App1, Ju1, JV}). Indeed such measures $\mu$ arise as the
invariant
measures of the Mehler semigroups of Example \ref{ex:1} (when these exist - see
e.g. \cite{CM, FR})
and in the case of (\ref{OUproc}), $\rho$ is the law of
$\int_{0}^{\infty}S(t-u)\,dL(u)$.
\end{example}

\begin{definition}\label{def:K}
 Let $\mu$ be a Radon probability measure on $E$ satisfying
$\wh{\mu}(x\s)\not=0$ for all $x\s\in E\s$.
For each Borel function $\phi:E\to \R$ we define the function
$K_{\mu,\phi}:E\to \C$ by $$K_{\mu,\phi}(x) := \frac{\exp(i\phi(x))}{\wh
\mu(\phi)}.$$
\end{definition}

We call $K_{\mu,\phi}$ an {\it exponential martingale vector}.

\begin{proposition}\label{prop:K} Let $\mu_1$ and $\mu_2$ be Radon probability
measures on $E_1$ and $E_2$,
respectively,  with
$\wh{\mu_2}(x\s)\not=0$ for all $x\s\in E_2\s$.
Let $T$ be a skew map with respect to the pair $(\mu_1,\mu_2)$.
Then for all $x\s\in E_2\s$
we have
\begin{align}\label{eq:identityPTK}
 P_TK_{\mu_2,x\s} = K_{\mu_1, x\s\circ T}.
\end{align}

\end{proposition}
\begin{proof}
Let $\rho$ denote the associated
skew convolution factor. From the identity
$$ \wh \mu_2(x\s) = \wh{T(\mu_1)}(x\s) \wh\rho(x\s) = \wh{\mu_1}(x\s\circ T)
\wh\rho(x\s)$$
we deduce that for all $x \in E$
$$
\begin{aligned}
P_T K_{\mu_2,x\s}(x)
& = \int_{E_2} K_{\mu_2,x\s}(T(x)+y)\,d\rho(y)
\\ & = \frac{\exp(i\la T(x), x\s\ra)}{ \wh \mu_2(x\s)}\int_{E_2} \exp(i\la y,
x\s\ra)\,d\rho(y)
\\ & = \frac{\exp(i\la T(x), x\s\ra)}{ \wh \mu_2(x\s)} \wh \rho(x\s)
 = \frac{ \exp(i( x\s\circ T)(x))}{\wh\mu_1( x\s\circ T)}
 = K_{\mu_1,x\s\circ T}(x).
\end{aligned}
$$
\end{proof}

Fix a Radon probability measure $\mu$ on $E$ and let $\mathscr{E}_{\mu}$ denote
the linear span of the set of
exponential martingale vectors $\{K_{\mu,x\s}; x\s\in E\s\}$.
The proposition
implies that, under the stated hypotheses on $\mu_1$, $\mu_2$, and $T$,
the mapping
$$ K_{\mu_2,x\s} \mapsto K_{\mu_1, x\s\circ T}$$
has a well-defined linear extension to a contraction from
$\mathscr{E}_{\mu_2}\to \mathscr{E}_{\mu_1}$
(this extension being also denoted $P_T$).
Under suitable assumptions on the measures one may
show that the functions $K_{\mu,x\s}$ are in fact linearly independent. This
fact is of some
interest by itself but is not needed
here; therefore we have included it in an appendix at the end of this paper.
Using the injectivity of the Fourier transform, a standard argument shows that
$\mathscr{E}_{\mu}$ is dense in $L^p(E,\mu;\C)$
for all $1\le p<\infty$ (see e.g. \cite[Lemma 5.3.1]{App2}), and consequently
$P_T$ is the
unique such extension.

\section{Second quantisation: The Gaussian case}\label{sec:Gaussian}

In this section we connect, in the Gaussian setting, the notions of skew
operators with
second quantisation. The presentation is slightly different from the usual one,
in that we introduce a form of the chaos expansion that utilises iterated
Malliavin
derivatives that was introduced by Stroock \cite{St}. This
approach will bring out the analogies between the Gaussian and the Poisson case
(which we present in the next section) very elegantly.

We begin by recalling some standard results from the theory of Gaussian
measures.
Proofs and more details can be found in the monographs \cite{Bo, Nu, VTC}.

Let $\mu$ be a Gaussian measure on the real
Banach space $E$, and let $H$ denote its reproducing
kernel Hilbert space,
which is defined as follows. The covariance operator $Q$ of $\mu$ is given by
$$ Qx\s = \int_E \lb x,x\s\rb x\,\mu(dx), \quad x\s\in E\s.$$
This integral is known to be absolutely convergent in $E$ and defines a bounded
operator $Q\in \calL(E\s,E)$
which is positive in the sense that $\lb Qx\s,x\s\rb \ge 0$ for all $x\s\in E\s$
and symmetric
in the sense that $\lb Qx\s, y\s\rb = \lb Qy\s, x\s\rb$ for all $x\s,y\s\in
E\s$.
The mapping $(Qx\s, Qy\s) \mapsto \lb Q x\s,y\s\rb$ defines an inner product on
the range of $Q$.
The real Hilbert space $H$ is defined to be the completion of the range of $Q$ with
respect to this
inner product. The identity mapping $Qx\s \mapsto Qx\s$ extends to a bounded
injective operator
$j: H\to E$, and we have the factorisation $Q = j\circ j\s$.
Here we have identified $H$ and its dual via the Riesz representation theorem.

Each element $h\in H$ of the form $h=j\s x\s$ defines a real-valued
function $\phi_h\in L^2(E,\mu)$ by
$ \phi_h(x) := \lb x,x\s\rb$,
and we have
$$ \n \phi_h\n_{L^2(E,\mu)}^2 = \int_E \lb x,x\s\rb^2\,d\mu(x)
= \n j\s x\s\n_{H}^2 = \n h\n_H^2.
$$
Since $j\s$ has dense range in $H$, the mapping
$h\mapsto \phi_h$ uniquely extends to an isometry from $H$ into $L^2(E,\mu)$.

\medskip
Suppose now that
$\mu_1$ and $\mu_2$ are Gaussian Radon measures on Banach spaces
$E_1$ and $E_2$, with reproducing kernel Hilbert
spaces $H_1$ and $H_2$ respectively.
In the next two Propositions \ref{prop:Gauss1} and \ref{prop:Gaussian-barT} we shall investigate
the relationship between linear skew maps from $E_1$ to $E_2$ with respect to
the pair $(\mu_1,\mu_2)$ and linear contractions from $H_1$ to $H_2$.

We begin by proving that if $T$
is a skew operator with respect to the pair $(\mu_1,\mu_2)$,
then $T$ restricts to a contraction between the reproducing kernel Hilbert
spaces.
This result and its proof extend a similar result for semigroup operators in
\cite{CMG1, VN1}.

\begin{proposition}\label{prop:Gauss1}
 If $T$ is a bounded linear operator from $E_1$ to $E_2$ which
is a skew operator with respect to the pair $(\mu_1,\mu_2)$
of Gaussian measures, then $T$ restricts to a contraction from $H_1$ to $H_2$.
\end{proposition}
\begin{proof}
By assumption we have $T\mu_1 * \rho = \mu_2$ for some Radon probability measure
$\rho$.
We claim that $\rho$ is Gaussian. Indeed, using the fact that $T\mu_1$ has mean
zero, we have
$$
\begin{aligned}
\int_E  \lb x,x\s\rb^2\, \mu_2(dx)
& =  \int_E \int_E  \lb Tx+y,x\s\rb^2\, \mu_1(dx)\,\rho(dy)\\ & =
\int_E  \lb x,x\s\rb^2\, T\mu_1(dx) + \int_E \lb y,x\s\rb^2\, \rho(dy).
\end{aligned}
$$
Hence, denoting the covariances of $\mu_1$ and $\mu_2$ by $Q_1$ and $Q_2$
(respectively), we see that
the operator $R:= Q_2 - TQ_1T\s$  is positive and symmetric as an operator from
$E_2\s$ to $E_2$.
Since $R\le Q_2$, a well-known tightness result for Gaussian measures implies
that
$R$ is the covariance of a Gaussian Radon measure $\tilde \rho$ on $E_2$. The
identity $TQ_1T\s + R = Q_2$
implies $T\mu_1* \tilde\rho = \mu_2$. Since $\mu_2$ is a  Gaussian measure, its
characteristic function
vanishes nowhere and hence, by the observation following Definition
\ref{def:skew}, $\rho = \tilde \rho$. This proves the claim.

Recall that $Q_1 = j_1\circ j_1\s$, where $j_1: H_1\embed E$ is the canonical
inclusion mapping,
and likewise we have
$Q_2 = j_2 \circ j_2\s$ and $R = j_R \circ j_R\s$.
For all $x\s\in E\s$ we have
\begin{equation}\label{eq:T1}
\begin{aligned}
\n j_1\s T\s x\s\n_{H_1}^2
&=\lb TQ_1 T\s x\s, x\s\rb
\\ &=\lb Q_2  x\s,  x\s\rb-\lb R\s x\s,  x\s\rb
\le \lb Q_2 \s x\s,  x\s\rb  = \n j_2\s x\s\n_{H_2}^2.
\end{aligned}
\end{equation}Hence,
\begin{equation}\label{eq:T2}
\begin{aligned}
|\lb Q_1 T\s x\s,  y\s\rb|
=|[ j_1\s T\s x\s, j_1\s y\s]_{H_1} |
\le \n j_2\s
x\s\n_{H_2} \n j_1\s y\s\n_{H_1}.\end{aligned}
\end{equation}
Define a linear functional  $\psi_{y\s}$ on the range of $ j_2\s$ by
$$\psi_{y\s}(j_2\s x\s) :=\lb Q_1 T\s x\s,  y\s\rb.$$
If $j_2\s x\s=0$, then $j_1\s T\s x\s=0$ by
\eqref{eq:T1}, so
$\psi_{y\s}$ is well-defined.
By \eqref{eq:T2}, $\psi_{y\s}$ extends to a
bounded linear functional on $H_2$ of norm $\le \n j_1\s y\s\n_{H_1}$.
Identifying $\psi_{y\s}$ with an element of $H_2$, for all $x\s\in E\s$
we have
$$\lb j_2 \psi_{y\s}, x\s\rb =
[ j_2\s x\s, \psi_{y\s}]_{H_2} = \lb Q_1 T\s x\s,  y\s\rb  =
\lb T Q_1\s y\s,x\s\rb.$$
Hence, $T Q_1 y\s = j_2 \psi_{y\s}$
and $\n T Q_1 y\s\n_{H_2}
\le \n
j_1\s y\s\n_{H_1}$.
Writing $Q_1 = j_1 j_1\s$ we see that the restriction of $T|_{H_1}$ to $H_1$
maps $j_1\s y\s $ to the element $ j_2 \psi_{y\s}$ of $H_1$, and that
$T|_H$ is contractive on the dense range of $j_1\s$ in $H_1$. This gives the
result.
\end{proof}

In the converse direction we have the following result.

\begin{proposition}\label{prop:Gaussian-barT}
 Suppose $T: H_1\to H_2$ is a linear contraction. Then $T$ admits a linear Borel
measurable
extension $\bar T: E_1\to E_2$ with the following properties:
\begin{enumerate}
 \item the image measure $\bar T(\mu_1)$ is a Gaussian Radon measure;
 \item there exists a Gaussian Radon measure $\rho$ on $E_2$ such that $\bar
T(\mu_1)*\rho = \mu_2$.
\end{enumerate}
In particular, $\bar T$ is a linear skew map for the pair $(\mu_1,\mu_2)$.
\end{proposition}
\begin{proof}
The following facts follows from the general theory of Gaussian measures (see,
e.g., \cite{Bo, FP}):
\begin{enumerate}
 \item the mapping $T:H_1\to H_2$ admits an extension to a linear Borel mapping
$\bar T: E_1\to E_2$;
 \item the operator $Q = j_2 T T\s j_2\s$ is the covariance of a Gaussian
measure
$\mu$ on $E_2$;
 \item $\mu$ coincides with the image measure $\bar T(\mu_1)$.
\end{enumerate}
In terms of
the covariance operators $Q$ and $Q_2$ of $\mu$ and $\mu_2$ we have
$$ \lb Q  x\s,x\s\rb = \n T\s j_2\s x\s\n_{H_1}^2 \le \n j_2\s x\s\n_{H_2}^2 =
\lb Q_2 x\s, x\s\rb.$$
Hence the positive symmetric operator
$ R:= Q_2 - Q$ is the covariance of a Gaussian measure $\rho$
for which we have
$ \bar T(\mu_1) * \rho  = \mu * \rho = \mu_2.$
\end{proof}

Our next objective is to relate the abstract second quantisation procedure of the previous
section to the Wiener-It\^o decompositions of $L^2(E_1,\mu_2)$ and $L^2(E_2,\mu_2)$.

Following the presentation in \cite{Nu}, for each $n\ge 1$ we define
 ${\H}_{n}$ to be the closed linear subspace of $L^2(E,\mu)$ spanned by
the functions $H_n(\phi_h)$, where $h\in H$ has norm one and $H_n$ is the $n$-th
Hermite polynomial given by the generating function expansion
$$ \exp\big(tx - \frac12t^2\big) = \sum_{n=0}^\infty \frac{t^n}{n!} H_n(x).$$
The Wiener-It\^o decomposition theorem
asserts that we have an orthogonal direct sum decomposition
$$
L^2(E,\mu)=\bigoplus_{n\ge 0}\H_n.
$$
Let $S_n$ be the permutation group on $n$ elements.
The range of the symmetrising projection $\Sigma_n: H^{\otimes n}\to H^{\otimes
n}$ defined by
$$ \Sigma_n(h_1\otimes \hdots \otimes h_n) : = \sum_{\sigma\in S_n}
(h_{\sigma(1)}\otimes \hdots \otimes h_{\sigma(n)})
$$
is denoted by $H^{\odot n}$ and is called the {\em $n$-fold symmetric tensor
product} of $H$.
Let $(h_n)_{n\ge 1}$ be an orthonormal basis of $H$ (the Hilbert space $H$,
being a reproducing kernel
Hilbert space of a Gaussian Radon measure, is separable (see e.g. \cite{Bo})).

Consider the $n$-fold stochastic integral $I_n: H^{\odot n}\to \H_n$, defined by
$$ I_n \big(\Sigma_n(h_{j_1}^{\otimes k_1}\otimes \hdots \otimes
h_{j_m}^{\otimes k_m})\big)
:=  \prod_{l=1}^m H_{k_l}(\phi_{h_{j_l}})$$
with $j_1< \dots < j_m$ and $k_1+\dots+k_m = n$. Then $\frac1{\sqrt{n!}} I_n$
sets up an isometric
isomorphism $H^{\odot n}\simeq \H_n$. Stated differently,
the mapping $I = \bigoplus_{n=0}^\infty \frac1{\sqrt{n!}} I_n$
defines an isometric isomorphism
$$ L^2(E,\mu) \simeq \Gamma(H),$$
where $$\Gamma(H):= \bigoplus_{n=0}^\infty H^{\odot n}$$
with norm $\n (h_n)_{n=0}^\infty\n_{\Gamma(H)}^2 = \sum_{n=0}^\infty\n h_n\n_{H^{\odot n}}^2$
 is the {\em symmetric Fock space} over $H$.

For a function $f: E\to \R$ of the form $$f = g(\phi_{h_1}, \dots,\phi_{h_n})$$
with $h_1,\dots,h_n$ orthonormal in $H$ and $g: \R^n\to \C$ of class $C^1$,
we define the {\em Malliavin derivative in the direction of $H$}
as the function $D f :E \to H$ given by
$$ D f = \sum_{j=1}^n \partial_j g(\phi_{h_1}, \dots,\phi_{h_n}) \otimes h_j.$$
As is well known (see e.g. \cite{Nu}), for all $1\le p<\infty$ the linear
operator $D$ is closable and densely defined
from $L^p(E,\mu)$ to $L^p(E,\mu;H)$.
From now on we will denote its closure by $D$ as well, and denote the domain of
its closure
by $W^{1,p}(E,\mu)$.
The higher order derivatives $D^k f: E\to H^{\otimes k}$ are defined recursively
by
$D^k f := D(D^{k-1}f)$. These operators are closable as well and the domains of
their closures
will be denoted by $W^{k,p}(E,\mu)$. We define the spaces $W^{\infty, p}(E,
\mu):
= \bigcap_{k \in \mathbb{N}}W^{k,p}(E,\mu)$.

The next proposition is due to Stroock \cite{St} in the context of an abstract
Wiener space.
We give a different proof for Gaussian measures on Banach
spaces.
We write $\E_\mu f = \int_E f\,d\mu$.

\begin{proposition}\label{prop:LP-Gaussian}
 The space $W^{\infty, 2}(E, \mu)$ is dense in $L^2(E,\mu)$  and for all $f \in
W^{\infty, 2}(E, \mu)$
we have
$$ f =  \sum_{n=0}^\infty \frac1{n!} I_n (\E_\mu D^n f).$$
\end{proposition}

\begin{proof}
For each $h\in H$,
the function $e_h: E\to \R$  is defined by $e_h := \exp(\phi_h - \frac12 \n
h\n_H^2)$. It is
well known that the linear span of $\{e_{h}, h \in E\}$ is dense in $L^{2}(E, \mu)$
(see e.g. \cite{Nu} for a proof). Since
$$ D^n e_h = e_h \otimes (\underbrace{h\otimes \hdots\otimes h}_{n \ {\rm
times}})$$ for all $\nN$, the first assertion follows. We clearly have
$$ \E_\mu D^n e_h = \E_\mu e_h \otimes (h\otimes \hdots\otimes h) = h\otimes
\hdots\otimes h.$$
Applying the $n$-fold stochastic integral and using the generating function
identity for the Hermite polynomials
 we obtain
\begin{align*}
\sum_{n=0}^\infty \frac1{n!} I_n (\E_\mu D^n e_h)
& = \sum_{n=0}^\infty  \frac1{n!}I_n(\underbrace{h\otimes \hdots\otimes h}_{n \
{\rm times}})
\\ & = \sum_{n=0}^\infty  \frac1{n!} H_n(\phi_h^n) = \exp\big(\phi_h-\frac12\n
h\n_H^2\big) = e_h,
\end{align*}
and the required result follows by density.
\end{proof}

Let us now return to the setting where $\mu_1$ and $\mu_2$ are
Gaussian measures on $E_1$ and $E_2$, having reproducing kernel Hilbert spaces $H_1$ and $H_2$,
respectively.
In order to avoid unnecessary notational complexity we will use the notation $D$
for Malliavin derivatives
acting on both $L^2(E_1,\mu_1)$ and $L^2(E_2,\mu_2)$, and define
$$D_h f := [Df, h].$$

\begin{lemma}\label{lem:PT-Gaussian}
Let $T: H_1\to H_2$ be a linear contraction. Then for all $f\in W^{n,2}(E_2,\mu_2)$ and
$h_1,\dots h_n\in H_1$,
$$
\E_{\mu_1} D_{h_1,\dots,h_n}^n P_T f = \E_{\mu_2} D_{Th_1, \dots,Th_n}^n f.
$$
\end{lemma}
\begin{proof} Let us check this first for $n=1$.
By an easy computation (see \cite{MN}), for $f\in W^{1,2}(E_2,\mu_2)$ we have
$P_T f\in W^{1,2}(E_2,\mu_2)$ and
$$ D P_T f = (P_T \otimes T\s)Df.$$
Consequently,
\begin{equation}
\label{eq:DhPT}
\begin{aligned}
 \E_{\mu_1} D_h P_T f
& = \E_{\mu_1} [(P_T \otimes T\s)Df, h]
\\ & =\int_E \int_E  [ Df(Tx + y), Th] \,d\rho(y)\,d\mu_1(x)
\\ & = \int_E [ Df(z), Th] \,d\mu_2(z)  = \E_{\mu_2}D_{Th} f.
\end{aligned}
\end{equation}
Here, $\rho$ is the measure constructed in Proposition \ref{prop:Gaussian-barT}.
The higher order case is proved along similar lines.
\end{proof}
In terms of the global derivative, the computation \eqref{eq:DhPT} shows that
$ \E_{\mu_1} D P_T f  = T\s  \E_{\mu_2} D f $
and more generally we have
\begin{equation} \label{stuff}
\E_{\mu_1} D^n P_T f = (T\s)^{\odot n} \E_{\mu_2} D^n f.
\end{equation}

Applying $I_n$ to both sides of (\ref{stuff}) and using Proposition
\ref{prop:LP-Gaussian} together with the density of $W^{\infty,2}(E,\mu)$ in $L^2(E,\mu)$
we have proved:

\begin{theorem}\label{thm:Gaussian}
 The following diagram commutes:
\begin{equation*}
  \begin{CD}
     L^2(E_2,\mu_2)     @> P_T >>  L^2(E_1,\mu_1)  \\
      @A  \bigoplus_{n=0}^\infty \frac1{\sqrt{n!}}I_n AA                @AA
\bigoplus_{n=0}^\infty\frac1{\sqrt{n!}} I_n  A \\
     \Gamma(H_2)     @> \bigoplus_{n=0}^\infty (T\s)^{\odot n} >>  \Gamma(H_1)
      \end{CD}
 \end{equation*}
\end{theorem}
The operator $ \Gamma(T\s)  := \bigoplus_{n=0}^\infty (T\s)^{\odot n}$
is usually called the {\em symmetric second quantisation} of the operator $T^{*}$.

\begin{remark}
The operator
$\bigoplus_{n=0}^\infty \frac1{\sqrt{n!}}I_n $ is inverse to $\bigoplus_{n=0}^\infty \frac1{\sqrt{n!}}\E_\mu D^n $
by  Proposition \ref{prop:LP-Gaussian},
so we may rewrite the commutative diagram in the following equivalent form:
 \begin{equation*}
  \begin{CD}
     L^2(E_2,\mu_2)     @> P_T >>  L^2(E_1,\mu_1)  \\
      @V  \bigoplus_{n=0}^\infty \frac1{\sqrt{n!}}\E_{\mu_2}D^n  VV
      @VV \bigoplus_{n=0}^\infty\frac1{\sqrt{n!}} \E_{\mu_1}D^n  V \\
     \Gamma(H_2)     @> \bigoplus_{n=0}^\infty (T\s)^{\odot n} >>  \Gamma(H_1)
      \end{CD}
 \end{equation*}
This diagram should be compared with the one in the next section.
\end{remark}

Let us finally return to the setting of the previous section and derive the
identity \eqref{eq:identityPTK} by the methods of the present section.
Fix $x\s\in E_2\s$ and let $h:= j_2\s x\s$. Then $K_{\mu_2,x\s} = \exp(i\phi_h - \n h\n_{H_2}^2)$
and therefore by Lemma \ref{lem:PT-Gaussian}
(which we apply to the real and imaginary parts of $K_{\mu_2,x\s}$), for all $g\in H_1$ we have
\begin{align*}
\E_{\mu_1} D_{g} P_T K_{\mu_2,x\s}
 = \E_{\mu_2} D_{Tg} K_{\mu_2.x\s}
 = i [h,Tg] \E_{\mu_2} K_{\mu_2,x\s}=   i [T\s h,g],
\end{align*}
so $\E_{\mu_1} D P_T K_{\mu_2,x\s} =  i  T\s h.$
Likewise we have
$$ \E_{\mu_1} D^n P_T K_{\mu_2,x\s} =i^n \otimes (\underbrace{T\s h\otimes \dots\otimes T\s h}_{n \
{\rm times}}).$$
Applying the $n$-fold stochastic integral, using Proposition \ref{prop:LP-Gaussian}
(again considering real and imaginary parts separately),
and using the (analytic extension of the) generating function
identity for the Hermite polynomials,
 we obtain
\begin{align*}
 P_T K_{\mu_2,x\s}
& = \sum_{n=0}^\infty \frac1{n!} I_n (\E_\mu D^n P_T K_{\mu_2,x\s} )
\\ & = \sum_{n=0}^\infty  \frac{i^n}{n!}I_n(\underbrace{T\s h\otimes \hdots\otimes T\s h}_{n \
{\rm times}})
\\ & = \sum_{n=0}^\infty  \frac{i^n}{n!} H_n(\phi_{T\s h}^n)
\\ & = \exp\big(i\phi_{T\s h}-\frac12\n
T\s h\n_H^2\big) = K_{\mu_1,T\s x\s},
\end{align*}
where the last identity used that $T\circ j_2 = j_2\circ T$ implies $T\s h = T\s j_2\s x\s = j_2\s T\s x\s.$

\section{Second quantisation: the Poisson random measure case}\label{sec:purejump}

We proceed with a similar result
in the case where $\mu_1$ and $\mu_2$ are
infinitely
divisible measures of pure jump type. For this we need do delve a bit deeper
into the structure
of such measures and develop their connection with Poisson random measures.

Let $(Y,\calY,\nu)$ be a $\sigma$-finite measure space and let $\N(Y)$ denote
the set of all
$\N$-valued measures on $Y$. We endow this space with the $\sigma$-algebra
$\sigma(\calY)$ generated
by $\calY$, that is, the smallest $\sigma$-algebra which renders the mappings
$\xi\mapsto \xi(B)$
measurable for all $B\in \calY$.

Let $(\Omega,\calF,\P)$ be a probability space and $\Pi$ be a
 Poisson random measure having intensity measure $\nu$. We denote by $\P_\Pi$
the
distribution of $\Pi$, that is,
$\P_\Pi$ is the probability measure on $(\N(Y), \sigma(\calY))$ given by
$$ \P_\Pi (B) = \P(\Pi\in B)$$
for measurable $B\subseteq \N(Y)$.

Following Last and Penrose \cite{LP}, for a measurable function
$f: \N(Y)\to \R$ and $y\in Y$ we define the measurable function
$D_y f: \N(Y)\to \R$ by
$$D_y f (\eta):= f(\eta+\delta_y) - f(\eta).$$
The function $D_{y_1,\hdots,y_n}^nf:\N(Y)\to\R$ is defined recursively by
$$ D^n_{y_1,\hdots,y_n}f = D_{y_n} D^{n-1}_{y_1,\hdots,y_{n-1}}f,$$ for $y_{1}, \ldots, y_{n} \in Y$.
This function is symmetric, i.e. it is invariant under any permutation of the variables.

We have a canonical isometry $ L_{\odot}^2(Y^n) = (L^2(Y))^{\odot n}$, where the former denotes
the closed subspace of $L^2(Y^n)$ comprised of all symmetric functions.
We set  $$\Gamma(L_{\odot}^2(Y)) = \bigoplus_{n=0}^\infty L_{\odot}^2(Y^n)$$
with norm $ \n (f_n)_{n=0}^\infty\n_H^2 = \sum_{n=0}^\infty \n f_n\n_{L_{\odot}^2(Y^n)}^2$;
for $n=0$ it is understood that
$(L^2(Y))^{\odot 0} = L_{\odot}^2(Y^0) := \R$. By
$I^n: L_{\odot}^2(Y^n)\to L^2(\Omega)$ we denote the
$n$-fold stochastic integral associated with $\Pi$ as defined in \cite{LP}.
We note that part (3) of the Last-Penrose theorem
is essentially a Stroock formula for
Poisson measures (cf. Proposition \ref{prop:LP-Gaussian}) and that a version of
this result for a class of real-valued L\'{e}vy processes may be
found in \cite{ESV}.

\begin{theorem}[Last-Penrose \hbox{\cite{LP}}]\label{thm:LP}
\begin{enumerate}
\item For all $\nN$, $ y_{1}, \ldots, y_{n} \in Y$, and  $f\in L^2(\P_\Pi)$ we have
$\tau^n f \in L_{\odot}^2(Y^n)$, where
$$ \tau^n f(y_1,\hdots,y_n) := \E D^n_{y_1,\hdots,y_n}f(\Pi).$$
\item The mapping  $ \tau:= \bigoplus_{n=0}^\infty \frac1{\sqrt{n!}}\tau^n$
is a surjective isometry from $L^2(\P_\Pi)$ onto
$ \Gamma(L_{\odot}^2(Y))$.
\item For all $f\in L^2(P_\Pi)$ we have
\begin{align*}
f(\Pi) = \sum_{n=0}^\infty \frac1{n!} I_n (\E D^n f(\Pi)).
\end{align*}
\end{enumerate}
\end{theorem}

From this point on, we shall consider the special case  $Y = E$, where $E$
is a separable real Banach space.
We use the shorthand notation
$$\bar{\Pi}(dx) = \one_{\{0 < \n x\n\le 1\}}{\wh\Pi}(dx) + \one_{\{\n x\n
>1\}}{\Pi}(dx),$$
where $\wh \Pi$ is the compensated Poisson random measure,
$$ \wh\Pi(B) = \Pi(B) - \nu(B),$$
and $\nu$ is now assumed to be a L\'{e}vy measure on $E$
(see e.g. \cite{Hey, Lin} for the definition).
We will need to use the L\'{e}vy-It\^{o} decomposition for Banach space-valued
L\'{e}vy processes,
as established in \cite{RvG}, and the next lemma is key in that
regard.

\begin{lemma}
The function $x\mapsto x $ is Pettis integrable with respect to $\bar \Pi$.
\end{lemma}

\begin{proof}
 Let $N$ be a Poisson random measure on $[0,\infty)\times E$ with intensity
measure $dt\times \nu$.
By a theorem of Riedle and Van Gaans \cite{RvG}, $x\mapsto x$ is Pettis
integrable with respect to $\bar N$.
It follows that $x\mapsto x$ is  Pettis integrable with respect to $\bar M$,
where $M$
is the Poisson random measure on $E$ given by $M(B) = N([0,1]\times B)$. Since
$M$ and $\Pi$
are identically distributed (both being Poisson random measures with intensity
measure $\nu$),
this proves the lemma.
\end{proof}

We will be interested in Borel probability measures $\mu$ on $E$ which arise as
the distribution of
$E$-valued random variables $X$ of the form
\begin{align}\label{eq:eq} X = \xi + \int_{E} x \,\bar{\Pi}(dx)
\end{align}
where $\Pi$ is a Poisson random measure on $E$ whose intensity measure $\nu$ is a L\'evy measure
and $\xi\in E$ is a given vector.
The interest of such random variables comes from the L\'evy-It\^o decomposition
for $E$-valued L\'evy processes, which
asserts that if $(L(t))_{t\ge 0}$ is a L\'evy process without Gaussian part,
then $L(1)$ is precisely of this form (see \cite{RvG}). Note that $\mu$ is a
Radon measure (since every Borel measure on a separable Banach space is Radon)
and infinitely divisible. In particular, its characteristic function vanishes
nowhere.

It will be convenient to define, for $f\in L^2(E,\mu)$,
\begin{align}\label{eq:D}
\tilde D_y f(x):= f(x+y) - f(x)
\end{align}
The higher order derivatives are defined recursively by
$\tilde{D}^n_{y_1,\dots,y_n}= \tilde{D}_{y_n}
\tilde{D}^{n-1}_{y_1,\dots,y_{n-1}}$.

\medskip
Suppose now that $\mu_1$ and $\mu_2$ are two measures of the above form,
associated with random variables
 $X_1:\Omega\to E_1$ and
$X_2:\Omega\to E_2$ which are given in terms of the vectors $\xi_1\in E_1$ and
$\xi_2\in E_2$ and Poisson random measures
$\Pi_1$ and $\Pi_2$ as in \eqref{eq:eq}.
 Consider a linear skew map $T:E_1\to E_2$
with respect to the pair $(\mu_1,\mu_2)$, so that
$$T\mu_1* \rho = \mu_2$$
for some unique Borel probability measure $\rho$.

We have the following analogue
of Lemma \ref{lem:PT-Gaussian}:

\begin{lemma}\label{lem:commute} 
For all
$f\in L^2(E_{2},\mu_2)$ and $y_1,\dots, y_n\in
E_{1}$,
\begin{align}\label{eq:higher-der-Poisson}
\E_{\mu_1} \tilde D_{y_1,\hdots,y_n}^n P_T f =\E_{\mu_2} \tilde
D_{Ty_1,\hdots,Ty_n}^n f.
\end{align}
\end{lemma}
\begin{proof}
Suppose the random variable $R: \tilde \Omega\to E_{2}$, defined on an
independent
probability space $(\tilde \Omega, \tilde \calF, \tilde P)$,  has
 distribution $\rho$. Then using the fact that $TX_1 + R$ and $X_2$ are
identically distributed,
\begin{align*}\E_{\mu_1} P_T f  =
\E \tilde\E f(TX_1 + R) = \E f(X_2) = \E_{\mu_2} f
\end{align*}
and
\begin{align*}
\E_{\mu_1} \tilde D_y P_T f & = \E_{\mu_1} P_T f(\cdot+y) - \E_{\mu_1} P_T
f(\cdot)
\\ & = \E\tilde \E f(Ty +TX_1 + R) - f(TX_1+R)
\\ & = \E  f(Ty + X_2)-f(X_2)
\\ & = \E_{\mu_2} f(\cdot+{Ty}) - \E_{\mu_2} f(\cdot)
\\ & = \E_{\mu_2} \tilde D_{Ty} f.
\end{align*}
For the higher derivatives we use a straightforward inductive
argument.
\end{proof}
Below we will think of the left and right hand side of \eqref{eq:higher-der-Poisson}
as symmetric functions on $E^n$. As such, the identity will be written as
$$  \E_{\mu_1}\tilde D^n P_T f
=  \E_{\mu_1}\tilde D^n f\circ T^{\odot n} ,$$
where
$$ (g\circ T^{\odot n})(y_1,\dots,y_n):= g(Ty_1,\dots,Ty_n).$$

\medskip
 Define, for $k=1,2$,
the operators $j_k: L^2(E_k, \mu_k)\to L^2(\P_{\Pi_k})$ by
 $$ j_k f(\eta) = f\Big(\xi + \int_{E} x \,\bar{\eta}(dx)\Big),
\quad \eta\in \N(E).$$
The rigorous interpretation of this identity is provided by noting that
$$ \n j_k f\n_{L^2(\P_\Pi)}^2 = \E \Big| f\Big(\xi +
\int_{E} x \,\bar{\Pi}(dx)\Big)\Big|^2 = \n
f\n_{L^2(E,\mu_k)}^2,$$
which means that $j_k f(\eta)$ is well-defined for $\P_\Pi$-almost all $\eta$
and
that $j_k$ establishes an isometry from $ L^2(E_k, \mu_k)$ into
$L^2(\P_{\Pi_k}).$
Note that
\begin{align}\label{eq:jPi}
j_k f (\Pi_k) = f(X_k)
\end{align}
and
$$ j_k \circ \tilde D = D \circ j_k,$$
and therefore, for all $g\in L^2(E_k,\mu_k)$,
$$ (\tau_k^n \circ j_k) g = \E D^n  j_k g(\Pi_k) = \E j_k \tilde D^n g (\Pi_k)=
 \E\tilde D^n  g(X_k) = \E_{\mu_k} \tilde D^n  g.  $$
Using this identity in combination with Lemma \ref{lem:commute}, for  all $f\in L^2(E_2,\mu_2)$ we obtain
$$(\tau_1^n \circ j_1) P_T f
=   \E_{\mu_1}\tilde D^n P_T f
=  \E_{\mu_1}\tilde D^n f\circ T^{\odot n}
= (\tau_1^n \circ j_1) f\circ T^{\odot n}.
$$
When combined with the contractivity of $P_T$
and the surjectivity of $\tau$ (see Theorem \ref{thm:LP}), this identity
implies that the mapping $f\mapsto f\circ T^{\odot n}$ is a
is a linear contraction from
$L_{\odot}^2(E_2^n,\nu_2^n)$ to $L_{\odot}^2(E_1^n,\nu_1^n)$.

In summary we have proved the following theorem.

\begin{theorem}\label{thm:Poisson}  Let $T:E_1\to E_2$ be a
linear Borel mapping. Under the above assumptions,
the mapping $(T^{*})^{\odot n}: f \mapsto f\circ T^{\odot n}$ is a linear contraction from
$L_{\odot}^2(E_2^n,\nu_2^n)$ to $L_{\odot}^2(E_1^n,\nu_1^n)$, and the following diagram commutes:
\begin{equation*}
  \begin{CD}
L^2(E_2,\mu_2)     @> P_T >>  L^2(E_1,\mu_1)  \\
    @V  \bigoplus_{n=0}^\infty \frac1{\sqrt{n!}}\E_{\mu_2}\tilde D^n   VV
    @VV \bigoplus_{n=0}^\infty \frac1{\sqrt{n!}}\E_{\mu_1}\tilde D^n   V \\
\Gamma( L^2(E_2,\nu_2))@ >\bigoplus_{n=0}^\infty (T^*)^{\odot n} >>\Gamma(L^2(E_1,\nu_1))
      \end{CD}
 \end{equation*}
\end{theorem}

To make the connection with the commuting diagram in the Gaussian case, which features the $n$-fold
stochastic integrals rather than $n$-fold derivatives, we note that by Theorem \ref{thm:LP} the following
diagram commutes as well for $k=1,2$:
\begin{equation*}
  \begin{CD}
L^2(E,\mu_k)     @> f\mapsto f(X_k) >>   L^2(\Omega) \\
    @V \bigoplus_{n=0}^\infty \frac1{\sqrt{n!}}\E_{\mu_k}\tilde D^n   VV
    @AA \bigoplus_{n=0}^\infty \frac1{\sqrt{n!}}I_n A \\
 \Gamma( L^2(E_k,\nu_k))@ > = >> \Gamma( L^2(E_k,\nu_k))
      \end{CD}
 \end{equation*}

Theorem \ref{thm:Poisson} is a generalisation of the result obtained by Peszat
\cite{Pesz} in the case where $\mu_1 = \mu_2$ is an invariant measure
associated with a Mehler semigroup on a Hilbert space $ E_1 = E_2$.
\smallskip

As we did in the previous section, we wish to make the link with the results
on skew operators. In principle we could repeat the Gaussian computation
at the end of Section \ref{sec:Gaussian}, but this requires the evaluation
of a rather intractable Poisson stochastic integral. There is, however, 
a simpler argument.

We start with some preparations. If $X$ and $\mu$ are as in \eqref{eq:eq},
then $K_{\mu,x\s} = \exp(i\lb x,x\s\rb - \zeta(x\s))$,
where $\wh{\mu}(x\s) = \exp(\zeta(x\s))$ is the characteristic function of
$\mu$ (with $\zeta$ the {\em L\'evy symbol} of $\mu$; see \cite[page 31]{App2}). Then,
for all $y\in E$ and for all $x\s\in E\s$,
\begin{align*}
\E_{\mu}\tilde D_{y} K_{\mu,x\s}
 & =  \E_{\mu}K_{\mu,x\s}(\cdot +y) -  \E_{\mu} K_{\mu,x\s}(\cdot)
\\ & = \exp(-\zeta(x\s))\E_{\mu} \exp(i\lb \cdot,x\s\rb)[\exp(i\lb y,x\s\rb)-1]
 = \exp(i\lb y, x\s\rb)-1 .
\end{align*}
Likewise, for $y_{1}, \ldots, y_{n} \in E$,
\begin{align}\label{eq:Prod-exp} \E_{\mu}\tilde D_{y_1,\dots,y_n}^n K_{\mu,x\s}
= \prod_{j=1}^n [\exp(i\lb y_j, x\s\rb)-1].
\end{align}

Now suppose that $T:E_1\to E_2$ is a skew operator with respect to $(\mu_1,\mu_2)$,
where the measures $\mu_k$ are the distributions of random variables $X_k$ as in \eqref{eq:eq} for $k=1,2$.
Then, by \eqref{eq:jPi}, the Last-Penrose theorem,
Lemma \ref{lem:commute} and \eqref{eq:Prod-exp}, for all $x\s\in E_2\s$ we have
\begin{align*}
P_T K_{\mu_2,x\s} (X_1) = j_1 P_T K_{\mu_2,x\s}(\Pi_{1})
 & =
\sum_{n=0}^\infty \frac1{n!} I_n (\E D^n j_1 P_T K_{\mu_2,x\s}(\Pi_1))
\\ & =
\sum_{n=0}^\infty \frac1{n!} I_n (\E_{\mu_1} \tilde D^n  P_T K_{\mu_2,x\s})
\\ & =
\sum_{n=0}^\infty \frac1{n!} I_n (\E_{\mu_2} \tilde D_{T\cdot}^n  K_{\mu_2,x\s})
\\ & = \sum_{n=0}^\infty \frac1{n!} I_n \Big(\prod_{j=1}^n [\exp(i\lb T\cdot_j, x\s\rb)-1]\Big)
\intertext{and, by duality and then repeating the same computation backwards,}
 & = \sum_{n=0}^\infty \frac1{n!} I_n \Big(\prod_{j=1}^n [\exp(i\lb \cdot_j, T\s x\s\rb)-1]\Big)
\\ & = \sum_{n=0}^\infty \frac1{n!} I_n (\E_{\mu_1} \tilde D^n  K_{\mu_1,T\s x\s})
\\ & = \sum_{n=0}^\infty \frac1{n!} I_n (\E D^n j_1 K_{\mu_1,T\s x\s}(\Pi_1))
\\ & = j_1 K_{\mu_1,T\s x\s}(\Pi_{1}) =
K_{\mu_1,T\s x\s}(X_1).
\end{align*}
It follows that $P_T K_{\mu_2,x\s}  = K_{\mu_1,x\s \circ T}$,
in agreement with \eqref{eq:identityPTK}.

\begin{remark}
The results of Sections \ref{sec:Gaussian} and \ref{sec:purejump} suggest the
problem of extending the theory to that case where $\mu_1$ and $\mu_{2}$ are
arbitrary infinitely divisible measures. We conjecture that Theorems \ref{thm:Gaussian}
and \ref{thm:Poisson} extend to this more general framework.
\end{remark}
\appendix
\section{Linear independence of the functions $K_{\mu,x\s}$}

The {\em support} of a Radon measure $\mu$ on  $E$ is the complement of the
union of all open $\mu$-null sets in $E$.
We denote the support of $\mu$ by ${\rm supp}(\mu)$ and its closed linear span
by $E_\mu$.
 We say that $\mu$ has {\em linear support} if ${\rm supp}(\mu) = E_\mu$. The
proof of the next result uses a variant of a standard technique of reduction to
a system of linear equations that can be found in \cite[pp. 20-21]{Gui} or
\cite[pp. 126-7]{Par}.

\begin{proposition}
Suppose that $\mu$ has linear support and let $F\subseteq E\s$ be such that its
points are separated by $E_\mu$.
Then the family $\{x\mapsto \exp(i\lb x,x\s\rb);x\s\in F\}$
is linearly independent in $L^2(E,\mu)$.
\end{proposition}
\begin{proof}
Let $x_{1}\s, \ldots x_{N}\s \in F$ be distinct linear functionals and let
$c_{1}, \ldots, c_{N} \in \R$ be such that
$\sum_{n=1}^{N} c_{n} \exp(i\lb \cdot,x_n\s\rb) = 0$ in $L^2(E,\mu)$.
In particular, the set $G$ of all $x\in E_\mu$ such that
$\sum_{n=1}^{N} c_{n} \exp(i\lb x,x_n\s\rb) = 0$ has full measure.
By the assumption on $\mu$, every open set $V$ which intersects $E_\mu$ has
positive $\mu$-measure and therefore
intersects $G$ in a set of positive $\mu$-measure.
It follows that $G$ is dense in $E_\mu$.
But then, by continuity, we find that $G=E_\mu$, that is, $\sum_{n=1}^{N} c_{n}
\exp(i\lb x,x_n\s\rb) = 0$ for all $x\in E_\mu$.
Hence $\sum_{n=1}^{N}{c_{n}}\exp(i t\la
x, x_n\s\ra) = 0$
for all $t \in \R$ and $x\in E_\mu$.
For $r=0,1,2,\hdots, N$ differentiate $r$ times with
respect to $t$ (where $1 \leq r \leq n-1$) and then put $t=0$. This yields
$\sum_{n=1}^{N}{c_{n}}\la x, x_n\s \ra ^{r}= 0$ for all $x\in E_\mu$.
We thus obtain a system of $N$ linear equations in ${c_{1}},
\ldots {c_{N}}$ and it has a non-zero solution if and only if
$$\left|
\begin{array}{c c c c} 1 & 1 & \cdots & 1 \\ \la x, x_{1}\s \ra  & \la x,
x_{2}\s \ra   & \cdots & \la x, x_{N}\s \ra \\
\cdot & \cdot & \cdots & \cdot  \\ \cdot & \cdot & \cdots & \cdot\\
\cdot & \cdot & \cdots & \cdot\\ \cdot & \cdot & \cdots & \cdot\\  \la x,
x_{1}\s\ra^{N-1} & \la x, x_{2}\s \ra^{N-1}
 & \cdots & \la x, x_{N}\s \ra^{N-1} \end{array}
\right| = 0.$$

  The left hand side of this equation is a Vandermonde determinant and so the
equation simplifies to
$$\prod_{1\leq m<n \leq N} (\la x, x_{m}\s \ra - \la x, x_{n}\s \ra) = 0.$$
Hence for each $x \in E$ there exist $1\le m,m\le N$ such that $\la x, x_{m}\s-
x_{n}\s \ra = 0$. The choice of
$m$ and $n$ here depends on $x$. We now prove that in fact they are independent
of the choice of
vector $x\in E_\mu$. To this end for each $1 \leq m,n \leq N$ define
$F_{mn}:=\{x \in E_\mu:\ \la x, x_{m}\s- x_{n}\s \ra = 0\}$.
Then $F_{mn}$ is closed and $\bigcup_{m,n=1}^{N}F_{mn} = E_\mu$.
By the Baire category theorem, at least one pair $(m,n)$
must be such that $F_{mn}$ has non-empty interior $O_{mn}$ in $E_\mu$. Let
$(m_0,n_0)$ be such a pair and
fix $x_{0} \in O_{m_0n_0}$.
Then by linearity $\la x - x_{0}, x_{m_0}\s - x_{n_0}\s \ra = 0$ for all $x \in
O_{m_0n_0}$. In other words
$\la y, x_{m_0}\s - x_{n_0}\s \ra = 0$ for all $y \in O_{m_0n_0}-\{x_{0}\}$. But
$O_{m_0n_0}-\{x_{0}\}$
contains an open
neighbourhood of $0$ in $E_\mu$
and hence by linearity again, $\la x, x_{m_0}\s- x_{n_0}\s \ra = 0$ for all $x
\in E_\mu$,
contradicting our assumptions. So we must have ${c_{1}} = \cdots = {c_{N}} = 0$.
\end{proof}

\vspace{5pt}

\begin{center} {\it Acknowledgment}
\end{center}
The authors are grateful to Josep Vives who pointed out to us the analogies between the
Last-Penrose and Stroock formulae and to Ben Goldys for helpful discussions.

\end{document}